\documentclass[12pt,a4paper,oneside]{amsart}

\usepackage[utf8]{inputenc}
\usepackage{amssymb}
\usepackage{amsthm}
\usepackage{bm} 
\usepackage[top=1.7cm, bottom=1.7cm, left=2.7cm, right=1.7cm]{geometry}
\usepackage[shortlabels]{enumitem}   
\setlist[itemize]{noitemsep,topsep=2pt,leftmargin=16pt}
\setlist[enumerate]{noitemsep,topsep=2pt,leftmargin=16pt} 

\usepackage{makecell}  
\usepackage{tabu}

\newcommand{\nn}{\mathbb{N}}

\newcommand\norm[1]{\left\lVert#1\right\rVert} 
\newcommand\dotprod[1]{\langle#1\rangle}       

\newtheorem{Def}{Definition}
\newtheorem{Lem}{Lemma}
\newtheorem{Prop}{Proposition}
\newtheorem{Thm}{Theorem}
\newtheorem{Cor}{Corollary}
\newtheorem{Rem}{Remark}

\title{On the extremal compatible linear connection of a Randers space}
\author{Csaba Vincze}
\address{Institute of Mathematics, University of Debrecen, H-4002 Debrecen, P.O.Box 400, Hungary}
\email{csvincze@science.unideb.hu}
\author{Márk Oláh}
\address{Institute of Mathematics, University of Debrecen, Doctoral School of Mathematical and Computational Sciences, H-4002 Debrecen, P.O.Box 400, Hungary}
\email{olma4000@gmail.com}

\begin{document}

\begin{abstract}
A linear connection on a Finsler manifold is called compatible to the metric if its parallel transports preserve the Finslerian length of tangent vectors. Generalized Berwald manifolds are Finsler manifolds equipped with a compatible linear connection. Since the compatibility to the Finslerian metric does not imply the unicity of the linear connection in general, the first step of checking the existence of compatible linear connections on a Finsler manifold is to choose the best one to look for. A reasonable choice is introduced in \cite{V14} called the extremal compatible linear connection, which has torsion of minimal norm at each point. 

Randers metrics are special Finsler metrics that can be written as the sum of a Riemannian metric and a 1-form (they are "translates" of Riemannian metrics). In this paper, we investigate the compatibility equations for a linear connection to a Randers metric. Since a compatible linear connection is uniquely determined by its torsion, we transform the compatibility equations by taking the torsion components as variables. We determine when these equations have solutions, i.e. when the Randers space becomes a generalized Berwald space admitting a compatible linear connection. Describing all of them, we can select the extremal connection with the norm minimizing property. As a consequence, we obtain the characterization theorem in \cite{Vin1}: a Randers space is a non-Riemannian generalized Berwald space if and only if the norm of the perturbating term with respect to the Riemannian part of the metric is a positive constant.
\end{abstract}

\maketitle
\footnotetext[1]{Cs. Vincze is supported by the EFOP-3.6.1-16-2016-00022 project. The project is co-financed by the European Union and the European Social Fund.}
\footnotetext[2]{Cs. Vincze and M. Oláh are also supported by TKA-DAAD 307818.}

\section{Notations and terminology}

Let $M$ be a differentiable manifold with local coordinates $u^1, \ldots, u^n.$ The induced coordinate system of the tangent manifold $TM$ consists of the functions $x^1, \ldots, x^n$ and $y^1, \ldots, y^n$. For any $v\in T_pM$, $x^i(v):=u^i\circ \pi (v)=u^i(p)=:p^i$ and $y^i(v)=v(u^i)$, where $i=1, \ldots, n$ and $\pi \colon TM \to M$ is the canonical projection.
 
A Finsler metric is a continuous function $F\colon TM\to \mathbb{R}$ satisfying the following conditions:  $\displaystyle{F}$ is smooth on the complement of the zero section (regularity), $\displaystyle{F(tv)=tF(v)}$ for all $\displaystyle{t> 0}$ (positive homogeneity) and the Hessian 
\[ g_{ij}=\frac{\partial^2 E}{\partial y^i \partial y^j} \]
of the energy function $E=F^2/2$ is positive definite at all nonzero elements $\displaystyle{v\in T_pM}$ (strong convexity). The pair $(M,F)$ is called a Finsler manifold.

\begin{Def} Let $(M,F)$ be a Finsler manifold. A linear connection $\nabla$ on $M$ is \emph{compatible} to $F$ if the parallel transports with respect to $\nabla$ preserve the Finslerian length of tangent vectors. If $\nabla$ is a compatible linear connection to $F$, then the triplet $(M,F,\nabla)$ is called a generalized Berwald manifold.
\end{Def}

Let $\nabla$ be a linear connection on $M$ compatible to $F$, which means that $F$ is constant along parallel vector fields. Denoting an arbitrary parallel vector field along the curve $c\colon [0,1]\to M$ by $X$,
\[ (F \circ X)'=(x^k\circ X)'{\frac{\partial F}{\partial x^k}}\circ X+(y^k \circ X)'{\frac{\partial F}{\partial y^k}}\circ X. \]
Here we have that $(x^k\circ X)'= (u^k \circ \pi \circ X)'= (u^k \circ c)' = {c^k}'$ and the function $y^k \circ X=: X^k$ satisfies the differential equation
\[ {X^k}' + X^{j} {c^i}' \Gamma_{ij}^{k} \circ c =0 \] 
of parallel vector fields, i.e. ${X^k}'=-{c^i}'  X^j  \Gamma_{ij}^k\circ c$. Therefore 
\[ (F \circ X)'={c^i}'\bigg(\frac{\partial F}{\partial x^i}-y^j {\Gamma}_{ij}^{k}\circ \pi \frac{\partial F}{\partial y^k}\bigg)\circ X. \]
So the compatibility of $\nabla$ is equivalent to
\begin{equation}
\label{compeqnatur}
\frac{\partial F}{\partial x^i}-y^j {\Gamma}^k_{ij}\circ \pi \frac{\partial F}{\partial y^k}=0 \hspace{1cm} (i=1,\dots,n).
\end{equation}
We will call (\ref{compeqnatur}) the system of \textbf{compatibility equations} and the vector fields 
\[ X_i^{h}:=\frac{\partial}{\partial x^i}-y^j {\Gamma}^k_{ij}\circ \pi \frac{\partial}{\partial y^k} \hspace{1cm} (i=1,\dots,n) \]
acting on the Finslerian metric in the compatibility equations are called \textbf{horizontal vector fields}. To sum up, the linear connection $\nabla$ determines (by its connection parameters) the horizontal vector fields $X_1^{h}, \dots, X_n^{h}$, and the vanishing of the derivatives of $F$ along them characterizes the compatibility of $\nabla$ to $F$. Together with $\nabla$ we also have a Riemannian metric on the manifold $M$, which is related to $F$ in a "nice way".

\begin{Def} Let $(M,F)$ be a Finsler manifold and suppose that $\gamma$ is a Riemannian metric on $M$. We will call $\gamma$ \emph{compatible} to $F$ if every linear connection compatible to $F$ is also compatible to $\gamma$.
\end{Def}

In other words, the parallel transports preserving the Finslerian length preserve the Riemannian length, too. Such a Riemannian metric always exists: it is well-known that the so-called averaged Riemannian metric given by integration over the indicatrices has this property; for the details see \cite{V5} and \cite{V11}. Unfortunately, it is hard to compute in practice, so instead (whenever possible) we should find a compatible Riemannian metric which is easier to handle. For some other candidates, see \cite{Mat1} and \cite{Cramp}. In what follows, $\gamma$ denotes an arbitrary (previously fixed) compatible Riemannian metric to $F$, $\nabla^*$ is its (uniquely existing) L\'{e}vi-Civita connection and ${\Gamma}^{k*}_{ij}$ are the (symmetric) connection parameters of $\nabla^*$. The horizontal vector fields generated by $\nabla^*$ are
\[ X_i^{h*}:=\frac{\partial}{\partial x^i}-y^j {\Gamma}^{k*}_{ij}\circ \pi \frac{\partial}{\partial y^k}  \hspace{1cm} (i=1,\dots,n). \]

\section{The compatibility equations in terms of the torsion components}

Following  \cite{V14}, we are going to transform the system \eqref{compeqnatur} of the compatibility equations for $\nabla$ by taking the torsion components as variables. The Christoffel process gives the following formula:
\begin{equation}
\label{torsion}
\Gamma_{ij}^r=\Gamma_{ij}^{*r}-\frac{1}{2}\left(T^{l}_{jk}\gamma^{kr}\gamma_{il}+T^{l}_{ik}\gamma^{kr}\gamma_{jl}-T_{ij}^r\right).
\end{equation}
In terms of the torsion components,  the system \eqref{compeqnatur} of the compatibility equations for $\nabla$ can be written into the form 
\[ X_i^{h^*}F+\frac{1}{2}y^j \left(T^{l}_{jk}\gamma^{kr}\gamma_{il}+T^{l}_{ik}\gamma^{kr}\gamma_{jl}-T_{ij}^r\right)\circ \pi \frac{\partial F}{\partial y^r}=0\ \  (i=1, \ldots,n). \]
Taking a nonzero vector $v \in T_p M$, we have the inhomogeneous system of linear equations 
\[ y^j (v)\left(T^{l}_{jk}\gamma^{kr}\gamma_{il}+T^{l}_{ik}\gamma^{kr}\gamma_{jl}-T_{ij}^r\right)(p) \frac{\partial F}{\partial y^r}(v)=-2X_i^{h^*}F(v) \hspace{1cm} (i=1, \dots, n) \]
for the unknown scalars $T_{ab}^{c}(p)$. The set $A_p(v)$ of its solutions is an affine subspace in the vector space $\wedge^2 T_p^*M \otimes T_pM$ spanned by 
\[ du^i_p \wedge du^j_p \otimes \left( \frac{\partial}{\partial u^k}\right)_p \quad (1\leq i < j\leq n, k=1, \ldots, n) \] 
of dimension $\displaystyle{\binom{n}{2}n}$. By going through all nonzero tangent vectors $v \in T_p M$, the solution set 
$$A_p=\bigcap_{v\in T_pM\setminus \{ \bf{0}\}} A_p(v)$$
containing the restrictions of the torsion tensors of all compatible linear connections to the Cartesian product $\displaystyle{T_pM\times T_pM}$ is also an affine subspace of $\wedge^2 T_p^*M \otimes T_pM$ (as the intersection of affine subspaces). Let us rewrite these equations in a more concise form. The skew-symmetry allows us to keep only the components $T_{ab}^{c}$ with lower indices $a<b$. Denoting by $\tilde{\sigma}_{ab;i}^{c}$ the coefficient of $T_{ab}^{c}$ ($a<b$) in the $i$-th equation,
\[ \tilde{\sigma}_{ab;i}^{c}=\left(y^a \gamma^{br}-y^b\gamma^{ar}\right)\frac{\partial F}{\partial y^r}\gamma_{ic}+\left(\delta_i^a \gamma^{br}-\delta_i^b\gamma^{ar}\right)\frac{\partial F}{\partial y^r}y^j\gamma_{jc}-\left(\delta_i^a y^b-\delta_i^b y^a\right)\frac{\partial F}{\partial y^c}. \]
If $\displaystyle{\partial/\partial u^1, \ldots, \partial/\partial u^n}$ is an orthonormal basis at $p\in M$ with respect to the compatible Riemannian metric $\gamma$, then
\begin{equation}
\label{coeffort}
\tilde{\sigma}_{ab;i}^{c}=\delta_i^c\left(y^a \frac{\partial F}{\partial y^b}-y^b\frac{\partial F}{\partial y^a}\right)+\delta_i^a \left(y^c\frac{\partial F}{\partial y^b}-y^b\frac{\partial F}{\partial y^c}\right)-\delta_i^b \left(y^c\frac{\partial F}{\partial y^a}-y^a\frac{\partial F}{\partial y^c}\right)
\end{equation}
and the compatibility equations are
\begin{equation}
\label{compeqinnprod}
 \sum_{a<b,c} \tilde{\sigma}_{ab;i}^{c} T_{ab}^{c} = -2X_i^{h^*}F \hspace{1cm} (i=1, \dots, n),
\end{equation}
where the summation symbol means summing over the following indices:
\[ \{ (a,b,c) \in \nn_n \times \nn_n \times \nn_n \, | \, a < b \} \hspace{1cm} (\nn_n:=\{1, \dots, n\}). \]

\section{The extremal compatible linear connection}
According to the previous section, the pointwise solutions of the compatibility equations form affine subspaces in $\wedge^2 T_p^*M \otimes T_pM$ ($p \in M$). So there is no reason for a global solution (if there is a solution at all) to be unique. Therefore the first step of checking the existence of compatible linear connections on a Finsler manifold is to choose the best one to look for. A  reasonable choice is introduced in \cite{V14}.

\begin{Def} Let $(M,F)$ be a Finsler manifold with a compatible Riemannian metric $\gamma$, and suppose that the coordinate vector fields $\partial/\partial u^1, \ldots, \partial/\partial u^n$ form an orthonormal basis at the point $p\in M$ with respect to $\gamma$. We introduce a Riemannian metric on $\wedge^2 T_p^*M \otimes T_pM$ in the following way.

If $\displaystyle{T=\sum_{i<j,k} T_{ij}^k du^i \wedge du^j \otimes \frac{\partial}{\partial u^k}}$, then
\begin{equation}
\label{torsionnorm}
\dotprod{T_p, S_p}: =\sum_{i<j, k} T_{ij}^k(p)S_{ij}^k(p) \hspace{1cm} \text{and}  \hspace{1cm} \norm{T_p}^2=  \sum_{i<j, k} {T_{ij}^k(p)}^2 .
\end{equation}
The extremal compatible linear connection on $M$ is the uniquely determined compatible linear connection whose torsion minimizes its pointwise norm.
\end{Def}

The system 
\[ du^i \wedge du^j \otimes \frac{\partial}{\partial u^k} \quad (1\leq i< j \leq n, k=1, \ldots, n) \]
is an orthonormal basis at the point $p\in M$ with respect to the scalar product introduced above. Since the pointwise solution set $A_p$ of the compatibility equations is an affine (especially, a convex) set, the element with the minimal norm (the closest element to the origin) is uniquely determined. For the details, see \cite{V14}.

\section{Compatible linear connections of Randers spaces}

\begin{Def} A Finsler metric is called \emph{Randers metric} if it has the special form
\[ F(x,y) = \alpha(x,y) + \beta(x,y),  \]
where
\begin{itemize}
\item $\alpha$ is a norm coming from a Riemannian metric on $M$, i.e. if the matrix of the Riemannian metric in a local basis $\partial/\partial u^1, \ldots, \partial/\partial u^n$ is $\alpha_{ij}$, then $\alpha(x,y)=\sqrt{\alpha_{ij}(x)y^{i}y^{j}}$,
\item $\beta$ is a 1-form, i.e. a linear functional on tangent spaces with the coordinate representation $\beta(x,y)=\beta_j(x)y^j$. 
\end{itemize}
\end{Def}

Note that both the metric components of the Riemannian part and the components of the perturbating term are considered on the tangent manifold as composite functions $\alpha_{ij}(x)$ and $\beta_k(x)$, where $x=(x^1, \ldots, x^n)$. It is well-known \cite{Vin1} that its Riemannian part is compatible to the Randers metric, so $\gamma:=\alpha$ is a convenient choice for a Riemannian environment because $\alpha$ is directly given by the Randers metric. 

In what follows, we are going to investigate the existence and the intrinsic expression of compatible linear connections on Randers spaces, i.e. when these spaces can be considered as generalized Berwald manifolds. We suppose that $M$ is connected and $\beta_p\neq 0$ at each point of the manifold. Otherwise, if there is a point $p \in M$ of a Randers manifold admitting compatible linear connections such that $\beta_p=0$, then the indicatrix at $p$ is quadratic and the same holds at all points of the (connected) manifold because the indicatrices are related by linear parallel translations. Therefore the metric is Riemannian, i.e. the compatibility of a linear connection means that it is metrical in the usual sense and the extremal compatible linear connection is obviously the L\'{e}vi-Civita connection of the Riemannian metric. 

\subsection{The compatibility equations}

Let a point $p\in M$ be given. In order to make the computations easier, we choose local coordinates around $p$ such that the basis $\partial/\partial u^1, \ldots, \partial/\partial u^n$ satisfies the following properties:
\begin{itemize}
\item $\displaystyle{(\partial/\partial u^1)_p, \ldots, (\partial/\partial u^n)_p}$ is an orthonormal basis of $T_p M$ with respect to $\alpha$, so $\alpha_{ij}(p)=\delta_{ij}$;
\item $\beta_1(p) = \dots = \beta_{n-1}(p)=0$ and $\beta_n(p) \neq 0$, i.e. the coordinate vector fields from $1$ to $n-1$ span the kernel of the linear functional  $\beta_p \colon T_pM \to \mathbb{R}$. 
\end{itemize}

Under these choices of the coordinate vector fields, the form of the metric at $p$ is
\[F(x,y)=\sqrt{\delta_{ij}y^i y^j} +\beta_n(x)y^n = \sqrt{(y^1)^2+\dots+(y^n)^2} + \beta_n(x)y^n. \]
The partial derivatives of $F$ with respect to the vectorial directions are
\[\dfrac{\partial F}{\partial y^k}(x,y)= \dfrac{y^k}{\sqrt{\sum_{i=1}^{n}(y^i)^2}} + \delta^n_k \beta_n(x)  \]
and, consequently, 
\[ y^a \frac{\partial F}{\partial y^b}-y^b\frac{\partial F}{\partial y^a} =
 y^a \left( \dfrac{y^b}{\sqrt{\sum_{i=1}^{n}(y^i)^2}} + \delta^n_b \beta_n \right) -
 y^b \left( \dfrac{y^a}{\sqrt{\sum_{i=1}^{n}(y^i)^2}} + \delta^n_a \beta_n \right) =\]
 \[\left(y^a \delta^n_b - y^b \delta^n_a \right)\beta_n. \]
Plugging it in formula \eqref{coeffort} and using that $a < b \leq n$ we have 
\[ \begin{gathered}
\tilde{\sigma}_{ab;i}^{c} = \left[
\delta_i^c \left(y^a \delta_b^n - y^b \delta_a^n \right)+
\delta_i^a \left(y^c \delta_b^n - y^b \delta_c^n \right)-
\delta_i^b \left(y^c \delta_a^n - y^a \delta_c^n \right) \right] \beta_n = \\
\left[ \left( -\delta_i^c y^b - \delta_i^b y^c \right) \delta_a^n +
\left( \delta_i^c y^a + \delta_i^a y^c \right) \delta_b^n +
\left( \delta_i^b y^a -\delta_i^a y^b \right) \delta_c^n \right] \beta_n = \\
 \left[ \left( \delta_i^c y^a + \delta_i^a y^c \right) \delta_b^n +
\left( \delta_i^b y^a -\delta_i^a y^b \right) \delta_c^n \right] \beta_n.
\end{gathered} \]
Writing $\sigma_{ab;i}^{c} := \tilde{\sigma}_{ab;i}^{c}/\beta_n$,
\begin{equation}
\label{coeffranders2}
\begin{gathered}
\sigma_{ab;i}^{c} = \left( \delta_i^c y^a + \delta_i^a y^c \right) \delta_b^n +
\left( \delta_i^b y^a -\delta_i^a y^b \right) \delta_c^n,
\end{gathered}
\end{equation} 
and the compatibility equations at a given point $p\in M$ are
\begin{equation}
\label{compeqinnprodranders}
 \sum_{a<b,c} \sigma_{ab;i}^{c} T_{ab}^{c} = - \frac{2}{\beta_n} X_i^{h^*} F \hspace{1cm} (i=1, \dots, n).
\end{equation}

\begin{Lem} \label{RHS} For a given point $p\in M$, 
\begin{equation} 
- \frac{2}{\beta_n} X_i^{h^*} F = 2 \left( {\Gamma}^{n*}_{ij}\circ \pi - \frac{1}{\beta_n} \frac{\partial \beta_j}{\partial x^i} \right)y^j \hspace{1cm} (i=1, \dots, n),
\end{equation}
where the right hand side is a linear expression in the coordinate functions $y^j$.
\end{Lem}

\begin{proof} Since $\nabla^*$ is the L\'{e}vi-Civita connection of $\alpha$, we have
\[ X_i^{h^*} F=X_i^{h^*}\alpha + X_i^{h^*} \beta = X_i^{h^*} \beta. \]
Furthermore,
\[ X_i^{h^*} \beta = \frac{\partial \beta_s y^s}{\partial x^i}-y^j {\Gamma}^{k*}_{ij}\circ \pi \frac{\partial \beta_s y^s}{\partial y^k} = \]
\[\frac{\partial \beta_s}{\partial x^i}y^s-y^j {\Gamma}^{k*}_{ij}\circ \pi  \beta_k = y^j \left(  \frac{\partial \beta_j}{\partial x^i} - \beta_n \Gamma^{n*}_{ij} \circ \pi  \right).  \qedhere \]
\end{proof}

\begin{Cor} Introducing the notations
\begin{equation} \label{C}
C_{j;i}:= \Gamma^{n*}_{ij}\circ \pi - \frac{1}{\beta_n} \frac{\partial \beta_j}{\partial x^i},
\end{equation}
the compatibility equations are
\begin{equation}
\label{compeqinnprodranders2}
\sum_{a<b,c} \sigma_{ab;i}^{c} T_{ab}^{c} = 2 C_{j;i}y^j \hspace{1cm} (i=1, \dots, n).
\end{equation}
\end{Cor}

\subsection{The coefficients of the torsion components} Let us arrange the components $T_{ab}^{c}$ into blocks such that all components with the same lower indices $(a,b)$ are in the same block, labeled by the index $(a,b)$. The number of blocks is $\displaystyle{\binom{n}{2}}$. In a given block, we order the components in an increasing way by the index $c$. 

\begin{Def} Let us call the first $\binom{n-1}{2}$ blocks whose indices $(a,b)$ does not contain the number $n$ \emph{front blocks}, and the remaining $n-1$ blocks  with indices $a < b=n$ \emph{rear blocks}. The last elements of the form $T_{ab}^{n}$ in the blocks are called \emph{tail}s, and a block without its tail is called a \emph{short block}.
\end{Def}

For example, in case of $n=4$, the components are arranged in the following way:

\begin{table}[h!]
\setlength{\tabcolsep}{2pt}
\bgroup
\def\arraystretch{1.5}
\begin{tabular}{|c|c|c|c|c|c|c|c|c|c|c|c|}
\hline 
\multicolumn{6}{|c|}{\text{front blocks}} &
\multicolumn{6}{|c|}{\text{rear blocks}} \\
\hline
\multicolumn{2}{|c|}{\text{block (1,2)}} &
\multicolumn{2}{|c|}{\text{block (1,3)}} &
\multicolumn{2}{|c|}{\text{block (2,3)}} &
\multicolumn{2}{|c|}{\text{block (1,4)}} &
\multicolumn{2}{|c|}{\text{block (2,4)}} &
\multicolumn{2}{|c|}{\text{block (3,4)}} \\
\hline 
\text{short b.} & \text{tail} &
\text{short b.} & \text{tail} &
\text{short b.} & \text{tail} &
\text{short b.} & \text{tail} &
\text{short b.} & \text{tail} &
\text{short b.} & \text{tail} \\
\hline
$T_{12}^{1} \, T_{12}^{2} \, T_{12}^{3}$ & $T_{12}^{4}$ &
$T_{13}^{1} \, T_{13}^{2} \, T_{13}^{3}$ & $T_{13}^{4}$ &
$T_{23}^{1} \, T_{23}^{2} \, T_{23}^{3}$ & $T_{23}^{4}$ &
$T_{14}^{1} \, T_{14}^{2} \, T_{14}^{3}$ & $T_{14}^{4}$ &
$T_{24}^{1} \, T_{24}^{2} \, T_{24}^{3}$ & $T_{24}^{4}$ &
$T_{34}^{1} \, T_{34}^{2} \, T_{34}^{3}$ & $T_{34}^{4}$ \\
\hline 
\end{tabular}
\egroup
\end{table}

Let us examine what the coefficients of the components $T_{ab}^{c}$ look like. By saying that a term appears in an equation we mean that its coefficient is not zero. Using \eqref{coeffranders2}, we have the following expressions:
\begin{itemize}
\item[I.] $\sigma_{ab;i}^{c}=0$ for any $T_{ab}^{c}$ ($a<b<n, c<n$) in a \textbf{front short block} because all indices are different from $n$.
\item[II.] $\sigma_{ab;i}^{n} = \delta_i^b y^a -\delta_i^a y^b$ for the \textbf{tail} $T_{ab}^{n}$ ($a<b<n$) \textbf{of a front block}. Therefore it appears only in two different equations of indices $i=a$ and $i=b$ with coefficients
\[ \sigma_{ab;a}^{n}=-y^b , \hspace{1cm} \sigma_{ab;b}^{n}=y^a. \]
\item[III.] $\sigma_{an;i}^{c} = \delta_i^c y^a + \delta_i^a y^c$ for any $T_{an}^{c}$ ($a<n, c<n$) in a \textbf{rear short block}.  We have two cases:
 \begin{itemize}
 \item if $a=c$, then the \textbf{diagonal component} $T_{an}^{a}$ appears only in the equation of index $i=a=c$ with coefficient
 \[ \sigma_{an;a}^{a} = 2y^a; \]
 \item if $a \neq c$, then $T_{an}^{c}$ appears in exactly two equations of indices $i=a$ and $i=c$ with coefficients
 \[ \sigma_{an;a}^{c}=y^c, \hspace{1cm}  \sigma_{an;c}^{c}= y^a. \]
 \end{itemize}
\item[IV.] $\sigma_{an;i}^{n} =  \delta_i^n y^a + \delta_i^a y^n + \delta_i^n y^a -\delta_i^a y^n  = 2\delta_i^n y^a$ for the \textbf{tail} $T_{an}^{n}$ ($a<n$)    \textbf{of a rear block}. So it only appears in the last equation of index $n$ with coefficient $\sigma_{an;n}^{n}= 2y^a$.
\end{itemize}

According to I-IV, none of the components of front short blocks appear in the equations. So we have the front tails and the rear blocks remaining, together
\[ \binom{n-1}{2}+(n-1)n = \frac{1}{2}(n-1)(3n-2) \]
tensor components. For example, in case of $n=2$ we have 2, in case of $n=3$ we have 7 and in case of $n=4$ we have 15 unknown components.

\subsection{The $n$-th compatibility equation} The $n$-th compatibility equation is much more simple than the others.

\begin{Lem}
\label{cool01} In the $n$-th compatibility equation, only the rear tails $T_{an}^{n}$ have coefficients different from $0$, and $\sigma_{an;n}^{n}= 2y^a$.
\end{Lem}

\begin{proof} We have seen in IV. that $\sigma_{an;n}^{n}= 2y^a$. What about the others?
\begin{itemize}
\item In short front blocks every coefficient is $0$.
\item The tails $T_{ab}^{n}$ of front blocks have nonzero coefficients in equations of indices $i=a$ and $i=b$, but here $a<b<n$.
\item Elements $T_{an}^{c}$ of rear short blocks have nonzero coefficients in equations of indices $i=a$ and $i=c$, but here $a<n$ and $c<n$. \qedhere
\end{itemize}
\end{proof}

\begin{Prop}
\label{cool02} The $n$-th compatibility equation divided by 2 is
\[ \sum_{a=1}^{n-1} y^a T_{an}^{n} = \sum_{a=1}^{n} C_{a;n} y^a. \]
\end{Prop}

\begin{proof} The left hand side is given by Lemma \ref{cool01}, and the right hand side is given by \eqref{compeqinnprodranders2}.
\end{proof}

\begin{Cor} \label{correartail} The $n$-th compatibility equation is solvable if and only if $C_{n;n}=0$ and the rear tails are uniquely determined:
\begin{equation}
\label{reartail}
T_{an}^{n}(p) = C_{a;n}(x) \overset{\eqref{C}}{=} {\Gamma}^{n*}_{an}(p) - \frac{1}{\beta_n(x)}  \frac{\partial \beta_a}{\partial x^n}(x)\hspace{1cm} (a=1, \dots, n-1).
\end{equation}
\end{Cor}

\begin{proof} By Proposition \ref{cool02}, both sides are linear expressions of the coordinate functions $y^a$. Comparing the coefficients of $y^a$ ($a \neq n$), we have both the solvability condition and formula (\ref{reartail}) immediately. 
\end{proof}

By Corollary \ref{correartail} we exploited all information the $n$-th equation has. From now on we consider the equations of indices $i=1, \dots, n-1$.

\subsection{The first $n-1$ compatibility equations} Let us consider the $i$-th compatibility equation ($i=1, \dots, n-1$) and rearrange its left hand side with respect to the coordinate functions $y^k$ ($k=1, \ldots, n$).

\begin{Lem}\label{cool03} For any $i=1, \ldots, n-1$, the $i$-th compatibility equation contains the coordinate function $y^k$ on the left hand side with coefficient 
$$\kappa_{k;i}=\left\{
\begin{array}{rl}
T_{ki}^{n}+T_{in}^{k} +T_{kn}^{i} & \textrm{if} \ \ k<i,\\
2 T_{in}^{i} & \textrm{if} \ \ k=i,\\
-T_{ik}^{n}+T_{in}^{k} +T_{kn}^{i} & \textrm{if} \ \ i<k<n,\\
0 & \textrm{if} \ \ $k=n$.
\end{array}
\right.$$
\end{Lem}

\begin{proof} First of all recall that the tensor components in the short front blocks do not appear in the compatibility equations because of I. Otherwise
\begin{itemize}
\item the tails $T_{ab}^{n}$ of the front blocks have nonzero coefficients in the equations of indices $i=a$ and $i=b$. So, in the $i$-th equation only the components of the form $T_{ib}^{n}$ and $T_{ai}^{n}$ appear with coefficients $\sigma_{ib;i}^{n}=-y^b$ and $\sigma_{ai;i}^{n}=y^a$.
 \begin{itemize}
 \item $y^i$ and $y^n$ never appear here because of $a<b<n$.
 \item Otherwise, $y^k$ appears with either $-T_{ik}^{n}$ (in case of $i<k$) or $T_{ki}^{n}$ (in case of $k<i$).
 \end{itemize}
\item The elements $T_{an}^{c}$ of the rear short blocks, which are not diagonal, have nonzero coefficients in the equations of indices $i=a$ and $i=c$ ($a \neq c$). So, in the $i$-th equation only the components of the form $T_{in}^{c}$ and $T_{an}^{i}$ appear with coefficients $\sigma_{in;i}^{c}=y^c$ and $\sigma_{an;i}^{i}=y^a$.
 \begin{itemize}
 \item $y^i$ and $y^n$ never appear here because of $a \neq c$, $a<n$ and $c<n$.
 \item Otherwise $y^k$ appears with $T_{in}^{k}+T_{kn}^{i}$.
 \end{itemize}
\item The rear diagonal component $T_{in}^{i}$ appears with $\sigma_{in;i}^{i}=2y^i$ in the $i$-th equation. Therefore $y^i$ appears with coefficient $2T_{in}^{i}$.
\item Rear tails do not appear here (only in the last equation). \qedhere
\end{itemize}
\end{proof}

\begin{Prop}\label{cool04} The $i$-th compatibility equation is 
\[ \sum_{k<i} (T_{ki}^{n}+T_{in}^{k} +T_{kn}^{i})y^k +  2T_{in}^{i} y^i + \sum_{i<k<n} (-T_{ik}^{n}+T_{in}^{k} +T_{kn}^{i})y^k = \sum_{k=1}^{n} 2 C_{k;i} y^k,  \]
where $i=1, \dots, n-1$.
\end{Prop}

\begin{proof} The left hand side is given by Lemma \ref{cool03}, and the right hand side is given by \eqref{compeqinnprodranders2}.
\end{proof}

\begin{Cor} \label{correardiag} For any $i=1, \ldots, n-1$, the $i$-th compatibility equation is solvable if and only if $C_{n;i}=0$ and the rear diagonal components are uniquely determined:
\begin{equation}
\label{reardiag}
T_{an}^{a}(p) = C_{a;a}(x) \overset{\eqref{C}}{=} {\Gamma}^{n*}_{aa}(p)-\frac{1}{\beta_n(x)} \frac{\partial \beta_a}{\partial x^a}(x)\hspace{1cm} (a=1, \dots, n-1).
\end{equation}
\end{Cor}

\begin{proof}  By Proposition \ref{cool04}, both sides are linear expressions of the coordinate functions $y^k$. Comparing the coefficients of $y^k$ ($k=1, \ldots, n-1)$, we have both the solvability condition and formula (\ref{reardiag}) immediately. 
\end{proof}

We have managed to express all tensor components having repeated indices. The remaining components have different indices one of which is $n$. By the condition $a<b$, every index-triplet appears exactly three times (so for $k<i<n$, we have the components $T_{ki}^{n},T_{in}^{k},T_{kn}^{i}$, one of which is a front tail and the other two are elements of the corresponding rear short blocks).

\begin{Cor} The front tails are uniquely determined:
\begin{equation}
\label{fronttail}
 T_{ab}^{n}(p)=\frac{1}{\beta_n(x)} \left( \frac{\partial \beta_b}{\partial x^a}(x) - \frac{\partial \beta_a}{\partial x^b}(x) \right),
\end{equation}
where $a<b<n$, and the components with the same indices in different positions satisfy the following equation:
\begin{equation}
\label{triplet2}
T_{an}^{c}(p) +T_{cn}^{a}(p) = 2 \Gamma^{n*}_{ac}(p) -\frac{1}{\beta_n(x)} \left( \frac{\partial \beta_c}{\partial x^a}(x) + \frac{\partial \beta_a}{\partial x^c}(x) \right)=:S_{ac}(x).
\end{equation}
\end{Cor}

\begin{proof} Let us consider the remaining components under the choice $k <i<n$. It is easy to see that component-triplets $T_{ki}^{n},T_{in}^{k},T_{kn}^{i}$ appear always together and exactly twice: 
\begin{itemize}
\item in the $i$-th equation we have $(T_{ki}^{n}+T_{in}^{k}+T_{kn}^{i})y^k$,
\item in the $k$-th equation we have $(-T_{ki}^{n}+T_{in}^{k}+T_{kn}^{i})y^i$.
\end{itemize}
By comparing them with the coefficients of $y^k$ and $y^i$ on the right hand side
\[ \begin{gathered}
\frac{\partial (i\textrm{-th eq.)}}{\partial y^k} \Rightarrow
 T_{ki}^{n}+T_{in}^{k} +T_{kn}^{i} = 2C_{k;i} \\
\frac{\partial (k\textrm{-th eq.)}}{\partial y^i} \Rightarrow -T_{ki}^{n}+T_{in}^{k}+T_{kn}^{i} = 2C_{i;k}.
\end{gathered}  \]
Adding these together, dividing by 2 and using the symmetry of $\Gamma^{k*}_{ij}$,
\[ \begin{gathered}
T_{in}^{k} +T_{kn}^{i} =  C_{k;i} + C_{i;k} \overset{\eqref{C}}{=} {\Gamma}^{n*}_{ik}\circ \pi - \frac{1}{\beta_n} \frac{\partial \beta_k}{\partial x^i} + {\Gamma}^{n*}_{ki}\circ \pi - \frac{1}{\beta_n} \frac{\partial \beta_i}{\partial x^k} = \\
2 \Gamma^{n*}_{ik} \circ \pi -\frac{1}{\beta_n} \left( \frac{\partial \beta_k}{\partial x^i} + \frac{\partial \beta_i}{\partial x^k} \right)=:S_{ik}.
\end{gathered} \]
Thus, from every component-triplet, the terms $T_{in}^{k}$ and $T_{kn}^{i}$ in the corresponding rear short blocks depend only on each other and their sum is $S_{ik}(x)$. Plugging this into the equation above, we can determine the front tails as
\[ T_{ki}^{n}= 2C_{k;i} - S_{ik} \overset{\eqref{C}}{=} 2{\Gamma}^{n*}_{ik}\circ \pi - \frac{2}{\beta_n} \frac{\partial \beta_k}{\partial x^i} - 2 \Gamma^{n*}_{ik}\circ \pi +\frac{1}{\beta_n} \left( \frac{\partial \beta_k}{\partial x^i} + \frac{\partial \beta_i}{\partial x^k} \right) \]
\[ = \frac{1}{\beta_n} \left( \frac{\partial \beta_i}{\partial x^k} - \frac{\partial \beta_k}{\partial x^i} \right). \qedhere \]
\end{proof}

\subsection{Solvability conditions} By Corollaries \ref{correartail} and \ref{correardiag}, there exists a solution of the compatibility equations at $p \in M$ if and only if
\[ C_{n;i} \overset{\eqref{C}}{=} {\Gamma}^{n*}_{in}\circ \pi - \frac{1}{\beta_n} \frac{\partial \beta_n}{\partial x^i}=0 \hspace{1cm} (i=1, \dots, n). \]
We are going to reformulate this condition in a coordinate-free way. In the following computations, the metric components of the Riemanian part and the components of the perturbating term are considered on the tangent manifold as composite functions $\alpha_{ij}(x)$ and $\beta_k(x)$, where $x=(x^1, \ldots, x^n)$. First, we compute the Christoffel symbols ${\Gamma}^{n*}_{in}$ of the L\'{e}vi-Civita connection $\nabla^*$ of $\alpha$. Using $\alpha_{ij}(p)=\delta_{ij}$, the compatibility to the metric implies that 
\begin{equation} \label{christoffel}
\frac{\partial}{\partial x^i} \alpha_{nn} = \frac{\partial}{\partial x^i} \alpha \left( \frac{\partial}{\partial x^n}, \frac{\partial}{\partial x^n} \right) = 2 \alpha \left( \Gamma_{in}^{l*}\circ \pi \frac{\partial}{\partial x^l}, \frac{\partial}{\partial x^n} \right) = 2 \Gamma_{in}^{l*}\circ \pi \alpha_{ln} = 2 \Gamma_{in}^{n*}\circ \pi.
\end{equation} 
Secondly, let us replace the 1-form $\beta$ with its dual vector $\beta^{\sharp}$:
\begin{equation} \label{dual} \frac{\partial \beta_n}{\partial x^i} = \frac{\partial (\beta^{l} \alpha_{ln})}{\partial x^i} = \frac{\partial \beta^{l}}{\partial x^i}  \alpha_{ln} + \beta^l \frac{\partial \alpha_{ln}}{\partial x^i} = \frac{\partial \beta^{n}}{\partial x^i} + \beta^n \frac{\partial \alpha_{nn}}{\partial x^i}.
\end{equation} 
Since $\alpha_{ij}(p)=\delta_{ij}$ implies that $\beta^n = \beta_n$, we have
\[ \begin{gathered}
\frac{\partial}{\partial x^i}(\beta^j \beta^k \alpha_{jk}) =
2 \frac{\partial \beta^j}{\partial x^i} \beta^k \alpha_{jk} + \beta^j \beta^k \frac{\partial \alpha_{jk}}{\partial x^i} =
2 \beta_n \frac{\partial \beta^n}{\partial x^i}  + \beta_n^2 \frac{\partial \alpha_{nn}}{\partial x^i} \overset{\eqref{dual}}{=} \\
2 \beta_n \left( \frac{\partial \beta_n}{\partial x^i} - \beta_n \frac{\partial \alpha_{nn}}{\partial x^i} \right)  + \beta_n^2 \frac{\partial \alpha_{nn}}{\partial x^i} =
2 \beta_n \frac{\partial \beta_n}{\partial x^i} - \beta_n^2 \frac{\partial \alpha_{nn}}{\partial x^i} \overset{\eqref{christoffel}}{=} \\
2 \beta_n \frac{\partial \beta_n}{\partial x^i} - 2 \beta_n^2 \Gamma_{in}^{n*} \circ \pi =
-2 \beta_n^2 \left( {\Gamma}^{n*}_{in}\circ \pi - \frac{1}{\beta_n} \frac{\partial \beta_n}{\partial x^i} \right).
\end{gathered} \]
So, we have proved the following characterization theorem obtained in \cite{Vin1} by different methods.

\begin{Prop} \label{existence} A connected non-Riemannian Randers space with a metric $F=\alpha+\beta$ is a generalized Berwald space if and only if the dual vector field of the perturbating term has a positive constant length with respect to the Riemannian part of the Randers metric.
\end{Prop}

\begin{proof} The metric admits a compatible linear connection if and only if the system \eqref{compeqinnprodranders2} has a solution for any $p\in M$. As we have seen above the solvability conditions
$$C_{n;i}(x) = {\Gamma}^{n*}_{in}(p) - \frac{1}{\beta_n(x)} \frac{\partial \beta_n}{\partial x^i}(x)=0\hspace{1cm} (i=1, \dots, n)$$
are equivalent to the vanishing of the partial derivatives (with respect to the position) of the norm square of the dual vector field with respect to the Riemannian part of the Randers metric. Using that the manifold is connected, we have that the norm must be constant. It is positive because the Randers metric is non-Riemannian.  
\end{proof}

\section{The extremal compatible linear connection of a Randers space}

\begin{Thm}\label{main} Let $F=\alpha+\beta$ be a non-Riemannian Randers metric on a connected manifold $M$. There exist a compatible linear connection on $M$ if and only if the dual vector field of the perturbating term has a positive constant length with respect to the Riemannian part of the Randers metric. If $p \in M$ is a given  point and $\partial/\partial u^1, \ldots, \partial/\partial u^n$ is a local coordinate system such that
\begin{itemize}
\item $(\partial/\partial u^1)_p, \ldots, (\partial/\partial u^n)_p$ is an orthonormal basis of $T_p M$ with respect to $\alpha$,
\item $\beta_1(p) = \dots = \beta_{n-1}(p)=0, \beta_n(p) \neq 0$, 
\end{itemize}
then the connection parameters of the extremal compatible linear connection are
\begin{equation}
\Gamma_{ij}^r(p)=\Gamma_{ij}^{*r}(p)-\frac{1}{2}\left(T^{i}_{jr}(p)+T^{j}_{ir}(p)-T_{ij}^r(p)\right),
\end{equation}
where $\Gamma_{ij}^{*r}$ denotes the connection parameters of the L\'{e}vi-Civita connection of $\alpha$, and the torsion components $T_{ab}^{c}$ of the extremal compatible linear connection are given as follows:
\begin{itemize}
\item all components in the front short blocks are $0$, 
\item the tails of the front blocks are
$$T_{ab}^{n}(p)=\frac{1}{\beta_n}(x) \left( \frac{\partial \beta_b}{\partial x^a}(x) - \frac{\partial \beta_a}{\partial x^b}(x) \right) \hspace{1cm} (a< b < n),$$
\item the rear diagonal elements are 
$$T_{an}^{a}(p) = {\Gamma}^{n*}_{aa}(p)-\frac{1}{\beta_n(x)} \frac{\partial \beta_a}{\partial x^a}(x)\hspace{1cm} (a=1, \dots, n-1),$$
\item the tails of the rear blocks are
$$T_{an}^{n}(p) = {\Gamma}^{n*}_{an}(p) - \frac{1}{\beta_n(x)}  \frac{\partial \beta_a}{\partial x^n}(x) \hspace{1cm} (a=1, \dots, n-1),$$
\item the non-diagonal elements of the rear short blocks are 
$$T_{an}^{c}(p)= \Gamma^{n*}_{ac}(p) -\frac{1}{2\beta_n(x)} \left( \frac{\partial \beta_c}{\partial x^a}(x) + \frac{\partial \beta_a}{\partial x^c}(x) \right) \hspace{1cm} (a, c=1, \dots, n-1, a \neq c).$$ 
\end{itemize}
\end{Thm}

\begin{proof} For the existence of compatible linear connections, we refer to Proposition \ref{existence}. To compute the connection parameters, we use formula \eqref{torsion} and $\alpha_{ij}(p)=\delta_{ij}$. 

In the previous section, we computed the components of the torsion of an arbitrary compatible linear connection. We get the extremal connection by choosing the free parameters in such a way that the quadratic sum of all the components is minimal.
\begin{itemize}
\item Components of the front short blocks never appear in the compatibility equations, thus they are free variables and no other components depend on them. So to minimize the norm, we set them all zero. 
\item The components $T_{ab}^{n}$, $T_{an}^{a}$ and $T_{an}^{n}$ are uniquely determined for every compatible linear connection, as we have seen by formulas \eqref{fronttail}, \eqref{reardiag} and \eqref{reartail}, respectively.
\item According to formula \eqref{triplet2}, the sum $T_{an}^{c}+T_{cn}^{a}$ ($a\neq c, c<n$) is the  constant $S_{ac}$ (and they only depend on each other). So the norm is minimized if we set them $T_{an}^{c}=T_{cn}^{a}= \frac{1}{2}S_{ac}$ by the solution of the minimum problem $x^2 + (c-x)^2 \to \mathrm{min}$. \qedhere
\end{itemize}
\end{proof}

\begin{Rem} {\emph{The so-called extremal linear connection can always be introduced on a Randers space by the formulas in Theorem \ref{main} for the components of the torsion tensor. However, it is not compatible to the Randers metric in general. The necessary and sufficient condition for such a linear connection to be compatible to the metric is that the dual vector field of $\beta$ has constant length with respect to the Riemannian metric $\alpha$.}}
\end{Rem}

\begin{Cor} According to the number of the tensor components that can be chosen arbitrarily, we have that
$$\dim A_p=n\binom{n-1}{2}$$
for a Randers space of dimension $n\geq 3 $ admitting compatible linear connections. Especially, it is of dimension zero, i.e. a singleton, in case of dimension two. 
\end{Cor}

\begin{proof} The number of the tensor components in the front short blocks is
$$(n-1)\binom{n-1}{2},$$
and the number of the components of the form $T_{an}^c$ ($a=1, \ldots, n-1, a\neq c, c< n$) is $\displaystyle{\binom{n-1}{2}}$. Their sum gives the dimension of the pointwise solution space $A_p$ of the compatibility equations.  
\end{proof}




\section{An example: the case of dimension 4}

In dimension $4$, the compatibility equation consists of $4$ equations with $\displaystyle{4\binom{4}{2}=24}$ unknown torsion components. Among the components in the $\displaystyle{\binom{3}{2}=3}$  front blocks only the tails have nonzero coefficients. Since $T_{12}^{1}, T_{12}^{2}, T_{12}^{3}; T_{13}^{1}, T_{13}^{2}, T_{13}^{3}; T_{23}^{1}, T_{23}^{2}, T_{23}^{3}$ do not appear in the equations (they have zero cefficients)  we set them all zero to get the extremal compatible connection. Using the formulas in I-IV., the matrix form of the compatibility equations is

\begin{table}[h!]
{\footnotesize
\tabulinesep=3pt
$\begin{tabu}{|c||ccc|cccc|cccc|cccc||c|}
\hline 
 & \multicolumn{3}{|c|}{\text{front tails}} & \multicolumn{4}{|c|}{\text{rear block (1,4)}} &  \multicolumn{4}{|c|}{\text{rear block (2,4)}} & \multicolumn{4}{|c||}{\text{rear block (3,4)}} & \\ 
\hline 
 & T_{12}^{4} &  T_{13}^{4} &  T_{23}^{4} &  T_{14}^{1} & T_{14}^{2} & T_{14}^{3} & T_{14}^{4} & T_{24}^{1} & T_{24}^{2} & T_{24}^{3} & T_{24}^{4} & T_{34}^{1} & T_{34}^{2} & T_{34}^{3} & T_{34}^{4} & \text{RHS} \\ 
\hline 
1 & -y^2 & -y^3  & 0 
    & \bm{ 2 y^1 } & y^2  & y^3  & 0
    & y^2  & 0 & 0 & 0
    & y^3  & 0 & 0 & 0 & 2C_{j;1} y^j \\
2 &y^1  & 0 & -y^3
    & 0 & y^1  & 0 & 0
    & y^1  & \bm{  2 y^2} & y^3  & 0
    & 0 & y^3 & 0 & 0 &  2C_{j;2} y^j \\ 
3 & 0 & y^1  & y^2
    & 0 & 0 & y^1  & 0 
    & 0 & 0 & y^2  & 0
    & y^1  & y^2 & \bm{2 y^3 } & 0 & 2C_{j;3} y^j \\ 
4 & 0 & 0 & 0
    & 0 & 0 & 0 & \bm{y^1 }
    & 0 & 0 & 0 & \bm{y^2 }
    & 0 & 0 & 0 & \bm{y^3 } & C_{j;4} y^j \\ 
\hline 
\end{tabu}$
\label{comptable1} }
\end{table}

According to Corollary \ref{correartail}, we can express the rear tails $T_{14}^{4}, T_{24}^{4}, T_{34}^{4}$ by comparing the coefficients of the $y^j$'s in the 4th equation:
\[ T_{14}^{4}(p) = C_{1;4}(x), \hspace{1cm}
   T_{24}^{4}(p) = C_{2;4}(x), \hspace{1cm}
   T_{34}^{4}(p) = C_{3;4}(x). \]

Rearranging the remaining terms in the first 3 equations, we get
\[ \begin{tabu}{ccccccc}
2T_{14}^{1} y^1 & + & (-T_{12}^{4}+T_{14}^{2}+T_{24}^{1})y^2 & + & (-T_{13}^{4}+T_{14}^{3}+T_{34}^{1})y^3 & = &  2C_{j;1} y^j   \\ 
&&&&&&\\
(T_{12}^{4}+T_{14}^{2}+T_{24}^{1}) y^1 & + & 2T_{24}^{2} y^2 & + & (-T_{23}^{4}+T_{24}^{3}+T_{34}^{2})y^3 & = &  2C_{j;2} y^j  \\  
&&&&&&\\
(T_{13}^{4}+T_{14}^{3}+T_{34}^{1}) y^1 & + & (T_{23}^{4}+T_{24}^{3}+T_{34}^{2}) y^2 & + & 2T_{34}^{3} y^3 & = &  2C_{j;3} y^j.  \\ 
\end{tabu}  \]

According to Corollary \ref{correardiag}, we can express the rear diagonal elements $T_{14}^{1}, T_{24}^{2}, T_{34}^{3}$ by comparing the coefficients of $y^i$ in the $i$-th equation:
\[ T_{14}^{1}(p) = C_{1;1}(x), \hspace{1cm}
   T_{24}^{2}(p) = C_{2;2}(x), \hspace{1cm}
   T_{34}^{3}(p) = C_{3;3}(x). \]

We can see that all triplets formed by the remaining terms appear twice. Let us consider for example  $T_{12}^{4},T_{14}^{2},T_{24}^{1}$. It appears in the first equation as the coefficient of $y^2$ and in the second equation as the coefficient of $y^1$. So we have
\[ \begin{tabu}{rcl}
-T_{12}^{4}+T_{14}^{2}+T_{24}^{1} & = &  2C_{2;1} \\
&&\\
T_{12}^{4}+T_{14}^{2}+T_{24}^{1}  & = &  2C_{1;2}
\end{tabu} \]
Adding these together and using \eqref{triplet2},
\[ T_{14}^{2}+T_{24}^{1} = 2 \Gamma^{4*}_{12} \circ \pi -\frac{1}{\beta_4} \left( \frac{\partial \beta_1}{\partial x^2} + \frac{\partial \beta_2}{\partial x^1} \right)=:S_{12}, \]
and by \eqref{fronttail},
\[ T_{12}^{4}= 2 C_{1;2} - S_{12} = \frac{1}{\beta_4} \left( \frac{\partial \beta_2}{\partial x^1} - \frac{\partial \beta_1}{\partial x^2} \right). \]
Using similar arguments, we can see that 
\[ T_{14}^{2}+T_{24}^{1} = S_{12}, \hspace{0.5cm}
T_{14}^{3}+T_{34}^{1} = S_{13}, \hspace{0.5cm}
T_{24}^{3}+T_{34}^{2} = S_{23}, \]
so in order to minimize the norm, we set them
\[ T_{14}^{2}=T_{24}^{1} = \frac{1}{2} S_{12}, \hspace{0.5cm}
T_{14}^{3}=T_{34}^{1} = \frac{1}{2} S_{13}, \hspace{0.5cm}
T_{24}^{3}=T_{34}^{2} = \frac{1}{2} S_{23}. \]
Finally,
\[ T_{13}^{4}= 2 C_{1;3} - S_{13} = \frac{1}{\beta_4} \left( \frac{\partial \beta_3}{\partial x^1} - \frac{\partial \beta_1}{\partial x^3} \right), \hspace{1cm}
T_{23}^{4}= 2 C_{2;3} - S_{23} = \frac{1}{\beta_4} \left( \frac{\partial \beta_3}{\partial x^2} - \frac{\partial \beta_2}{\partial x^3} \right). \]

\end{document}